\documentclass[11pt,twoside,a4paper]{amsart}
\usepackage[french]{babel}
\usepackage{amssymb,bbm}
\usepackage{amsmath}
\usepackage{amsthm}
\usepackage[dvips]{graphics}
\usepackage{graphicx}
\usepackage{latexsym}
\usepackage[normal]{subfigure}
\input{epsf.sty}
\include{psfig}
\usepackage[all]{xy}
\usepackage{eufrak}
\usepackage{mathrsfs}
\usepackage{lipsum}
\usepackage{perpage}
\usepackage{amsfonts}
\usepackage{hyperref}
\begin{document}
\def\K{\mathbb{K}}
\def\R{\mathbb{R}}
\def\C{\mathbb{C}}
\def\Z{\mathbb{Z}}
\def\Q{\mathbb{Q}}
\def\D{\mathbb{D}}
\def\N{\mathbb{N}}
\def\T{\mathbb{T}}
\def\P{\mathbb{P}}
\def\A{\mathscr{A}}
\def\CC{\mathscr{C}}
\renewcommand{\theequation}{\thesection.\arabic{equation}}
\newtheorem{theorem}{Th\'eor\`eme}[section]
\newtheorem{cond}{C}
\newtheorem{lemma}{Lemme}[section]
\newtheorem{corollary}{Corollaire}[section]
\newtheorem{prop}{Proposition}[section]
\newtheorem{definition}{D\'efinition}[section]
\newtheorem{remark}{Remarque}[section]
\newtheorem{example}{Exemple}[section]
\bibliographystyle{plain}

\title{R\'esolution du $\partial\bar\partial$ pour les courants prolongeables d\'efinis dans un anneau}
\author[ E.\  Bodian \& I.\  Hamidine \\  \& S.\  Sambou ]
{Eramane Bodian \& Ibrahima Hamidine\\  \& Salomon Sambou}

\address{D\'epartement de Math\'ematiques\\UFR des Sciences et Thechnologies \\ Universit\'e Assane Seck de Ziguinchor, BP: 523 (S\'en\'egal)}

\email{m.bodian2966@zig.univ.sn \& i.hamidine5818@zig.univ.sn \& ssambou@univ-zig.sn }


\date{\today}

\maketitle
\begin{abstract}
Dans ce papier, on r\'esout d'abord le $\partial\bar\partial$ pour les courants prolongeables d\'efinis dans $\C^n$ priv\'e d'une boule $B$ de $\C^n$, ensuite dans une vari\'et\'e analytique complexe $X$, on le r\'esout pour un domaine $D=X\setminus\bar{\Omega}$, o\`u $\Omega$ est un domaine born\'e de $X$ d\'efini par $\{z\in X\,\ /\,\ \varphi(z)<0\}$, (avec $\varphi$ une fonction d'exhaustion  strictement plurisousharmonique).
\vskip 2mm
\noindent
{\normalsize A}{\tiny BSTRACT.} In this present paper, we first solve  the $\partial\bar\partial$ for  extendable currents defined in  $\C^n\setminus B$, where $B$ is a ball  of $\C^n$, then in a analytic complex manifold  $X$,  and in a domain $D=X\setminus\bar{\Omega}$ where $\Omega$ is a bounded domain of $X$ defined by $\{z\in X\,\ /\,\ \varphi(z)<0\}$, ($\varphi$ is an exhaustion strictly plurisubharmonic function).   
\vskip 2mm
\noindent
\vskip 1.3mm
\noindent
\textit{Classification math\'ematique 2010~:}  32F32.
\end{abstract}

%\end{center}
\section*{Introduction}

Soit $B\subset\C^n$ la boule unit\'e, on se pose la question suivante~: si $T$ est un courant prolongeable $d$-ferm{\'e} sur $\C^n\setminus \bar B$, existe-t-il un courant prolongeable sur $\C^n\setminus \bar B$ tel que $\partial \bar{\partial} u = T$?
\vskip 1mm
\noindent
Tenant compte de consid{\'e}rations classiques, nous devons pour r{\'e}pondre {\`a} cette question, avoir {\`a} r{\'e}soudre  l'{\'e}quation
\begin{equation}\label{Equa1}
d u =T,
\end{equation}
 o{\`u} $T$ est un courant prolongeable, la solution obtenue se d{\'e}compose sans perte de g\'en\'eralit\'es en une partie $\partial$-ferm{\'e}e et l'autre $\bar{\partial}$-
ferm{\'e}e.  $\C^n\setminus \bar B$ a les conditions g\'eom\'etriques n\'ecessaires \`a la r\'esolution du $\partial$ et $\bar \partial$ pour les courants prolongeables (voir \cite{Samb}).  Partant de r{\'e}sultats
connus de cohomologie de de Rham et de l'analogue convexe (voir \cite{SBD}), alors l'{\'e}quation
\eqref{Equa1} admet une solution. La r{\'e}solution du $\partial \bar{\partial}$
devient alors une cons{\'e}quence des r{\'e}sultats de r{\'e}solution du
$\bar{\partial}$ pour les courants prolongeables obtenus dans \cite{Samb}. Dans le cas d'une vari\'et\'e, on introduit la notion d'extension contractile et on r\'esout le $\partial\bar{\partial}$ dans ce cadre.
\section{Pr{\'e}liminaires et notations}
Soit $B$ la boule  de $\C^n$.
\begin{definition}$~~ $\\
Un courant $T$ d\'efini sur $\C^n\setminus \bar B$ est dit prolongeable s'il existe un courant $\check{T}$ d\'efini sur $\C^n$ tel que $\check{T}_{|(\C^n\setminus \bar B)}=T$.
\end{definition}
   D'apr\`es Martineau \cite{Mart}, puisque $\stackrel{\circ}{(\overline{\C^n\setminus \bar B})}=\C^n\setminus \bar B$,  les courants prolongeables de degr\'e $p$ sur $\C^n\setminus \bar B$ sont \'egaux au dual topologique des $(2n - p)$-formes diff{\'e}rentielles de classe $\CC^{\infty}$ sur $\C^n$ \`a support compact sur $\C^n\setminus B$. On note $\check{\mathcal{D}}'^p(\C^n\setminus \bar B)$ l'espace des  $p$-courants d\'efinis sur $\C^n\setminus \bar B$ et prolongeables \`a $\C^n$, $\A_c^p
(\C^n\setminus \bar B)$ les $p$-formes diff{\'e}rentielles de classe
$\CC^{\infty}$ sur $\C^n$ \`a support compact dans $\C^n\setminus \bar{B}$. Sur $\C^n$, on note $\check{\mathcal{D}}'^{p, q}
(\C^n\setminus \bar{B})$ l'espace des $(p, q)$-courants prolongeables d{\'e}finis sur
$\C^n\setminus \bar{B}$ et $\A_c^{p, q} (\C^n\setminus \bar{B})$ l'espace des $(p, q)$-formes
diff{\'e}rentielles {\`a} support compact dans $\C^n\setminus \bar{B}$. On note $\check{{\rm H}}^p(\C^n\setminus \bar{B})$ le $p^{{\rm ieme}}$ groupe de cohomologie de de Rham des courants
prolongeables d{\'e}finis sur $\C^n\setminus \bar{B}$, $\check{{\rm H}}^{p, q} (\C^n\setminus \bar{B})$ le $(p,q)^{{\rm ieme}}$ groupe de cohomologie de Dolbeault des courants prolongeables
d{\'e}finis sur $\C^n\setminus \bar{B}$. Si $F \subset \C^n$, alors ${\rm H}_{\infty}^p (F)$
d{\'e}signe le $p^{{\rm ieme}}$ groupe de cohomologie de de Rham des
$p$-formes diff{\'e}rentielles de classe $\CC^{\infty}$ d{\'e}finis
sur $\C^n$, ${\rm H}_{\infty, c}^p (\C^n)$ est le groupe de cohomologie de de Rham des
$p$-formes diff{\'e}rentiables de classe $\CC^{\infty}$ sur $\C^n$ {\`a}
support compact et enfin $\A^p (F)$ l'espace des $p$-formes diff{\'e}rentielles de classe $\CC^{\infty}$ sur $F$. On note aussi, pour tout domaine $D$ de $\C^n$, $\flat D$ le bord de $D$.
\section{R\'esolution de l'\'equation $du=T$}
Tout le long de cette section, nous consid\'erons
\[
S=\{z\in\C^n\, ; |z|=1\}
\]
et
\[
B=\{ z\in \C^n\, ; |z|<1\},
\]
 la sph\`ere et la boule unit\'e respectivement dans $\C^n$.
\vskip 1mm
\noindent
Consid\'erons la suite courte suivante, pour $0\leq p\leq 2n$
\[
0\to \A^p(\C^n)\to \A^p(\C^n\setminus \bar B)\oplus\A^p(\bar B)\to \A^p(S)\to 0.
\]
Sur le plan de cohomologie de de Rham, on a la suite longue de cohomologie suivante~:
\begin{equation*}\label{suitelong1}
\begin{split}
& 0 \to {\rm H}^0(\C^n)\to {\rm H}^0(\C^n\setminus \bar B)\oplus {\rm H}^0(\bar B)\to{\rm H}^0(S)\to\\
&  {\rm H}^1(\C^n)\to {\rm H}^1(\C^n\setminus \bar B)\oplus {\rm H}^1(\bar B)\to{\rm H}^1(S)\to \ldots\to \\
& {\rm H}^{2n-1}(\C^n)\to {\rm H}^{2n-1}(\C^n\setminus \bar B)\oplus {\rm H}^{2n-1}(\bar B)\to {\rm H}^{2n-1}(S)\\
&\to {\rm H}^{2n}(\C^n)\to {\rm H}^{2n}(\C^n\setminus \bar B)\oplus {\rm H}^{2n}(\bar B)\to 0
\end{split}
\end{equation*}
On sait que
\[
{\rm H}^p(\C^n\setminus \bar B)={\rm H}^p(\C^n\setminus B).
\]
\begin{remark}[\cite{God}]\label{remark1}
$$
\left\{
  \begin{array}{ll}
    {\rm H}^p(S)=0, & \hbox{ si }\, 1< p < 2n-1\\
    {\rm H}^0(S)={\rm H}^{2n-1}(S)=\R, \\
    {\rm H}^p(\R^{2n})=0, & \hbox{ si }\, p\geq 1\\
    {\rm H}^0(\R^{2n})=\R,  \\
   {\rm H}^p(B)=0, & \hbox{ si }\, p\geq 1\\
    {\rm H}^0(B)=\R.
  \end{array}
\right.
$$
\end{remark}
\begin{theorem}\label{thm1}$~~ $\\
\[ \check{{\rm H}}^j(\C^n\setminus\bar B)=0 \mbox{ pour } 2\leq j\leq 2n-2.
\]
\end{theorem}
Pour d\'emontrer le th\'eor\`eme \ref{thm1}, on a besoin du lemme suivant~:
\begin{lemma}\label{lem1}$~~ $\\
$\A_c^p(\C^n\setminus B)\cap \ker d=d(\A_c^{p-1}(\C^n\setminus B)\big)$  pour $2\leq p\leq 2n-1$.
\end{lemma}
\begin{proof}$~~ $\\
On utilse les r\'esultats suivants:
\[
{\rm H}_c^p(\C^n)=0, \mbox{ si } p \leq 2n-1
\]
\[
{\rm H}_c^{2n}(\C^n)=\R, \mbox{ pour } p=2n.
\]
Si $f\in\A_c^p(\C^n\setminus B)\cap \ker d$, alors $f\in\A_c^p(\C^n)\cap \ker d$ si $1\leq p\leq 2n-1$.
\vskip 1mm
\noindent
${\rm H}_c^{p}(\C^n)=0$, alors il existe $u\in \A_c^{p-1}(\C^n)$ telle que $du=f$.
\vskip 1mm
\noindent
Si $p=1$, $u$ est une $0$-forme diff\'erentielle \`a support compact. Alors $du_{|_B}=0$. Ainsi $u=cst$ sur $B$. Par suite $du_{|_{\C^n\setminus\bar{B}}}\not\in \A_c^p(\C^n\setminus B)$ sauf pour $u$ identiquement nulle.
\vskip 1mm
\noindent
Si $p\geq 2$, $du_{|_B}=0$. Puisque 
\[
{\rm H}^{p-1}(B)={\rm H}^{p-1}(\bar{B})=0, \mbox{ pour } 1\leq p-1;
\]
 i.e; $p\geq 2$, il existe $v\in \A^{p-2}(\bar{B})$ tel que $dv=u$ sur $B$. Posons $\tilde{v}$ une extension \`a support compact dans $\C^n$ de $v$, on a
\[
\tilde{u}=u-d\tilde{v}
\]
qui est un \'el\'ement de $\A_c^{p-1}(\C^n\setminus  B)$ tel que $d\tilde{u}=f$.
\end{proof}
\begin{proof}[D\'emonstration du th\'eor\`eme \ref{thm1}]$~~ $\\
\textit{\'Etape 1:}
Soit $T\in \check{\mathcal{D}'}^p(\C^n\setminus \bar{B})\cap\ker d$, $2\leq p\leq 2n-2.$
\vskip 1mm
\noindent
L'espace $d\A_c^{2n-p}(\C^n\setminus \bar{B})$ est ferm\'e dans $\A_c^{2n-p+1}(\C^n\setminus B)$ pour $2\leq 2n-p+1\leq 2n-1$, (voir par exemple \cite{Samb}, remarque 2). 
\vskip 1mm
\noindent
Pour $K$ un compact de $\C^n\setminus B$, notons 
 $\A_{c,K}^{2n-p+1}(\C^n\setminus B)$ le sous-espace des formes diff\'erentielles appartenant \`a  $\A_c^{2n-p+1}(\C^n\setminus  B)]$ et qui ont leur support dans $K$.
 \vskip 1mm
 \noindent
 L'espace $\A_{c,K}^{2n-p+1}(\C^n\setminus B)\cap d\A_c^{2n-p}(\C^n\setminus  B)$ est ferm\'e dans $\A_{c,K}^{2n-p+1}(\C^n\setminus B)$ qui est un espace de Fr\'echet, par cons\'equent, c'est un espace de Fr\'echet.
\[
\A_{c,K}^{2n-p+1}(\C^n\setminus B)\cap d\A_c^{2n-p}(\C^n\setminus  B)=\bigcup_{\nu\in\N}\big(\A_{c,K}^{2n-p+1}(\C^n\setminus B)\cap d\A_{K_\nu}^{2n-p}(\C^n\setminus B)\big);
\]
 avec $K_\nu=\{z\in\C^n \; |z|\leq R_\nu,\,\ R_\nu\in\R_+^*\}\setminus B $ et $R_\nu > 1$ une suite exhaustive de compacts dans $\C^n\setminus B$. Il existe $\nu_0$ tel que $\A_c^{2n-p+1}(\C^n\setminus B)\cap d\A_{K_{\nu_0}}^{2n-p}(\C^n\setminus B)$ soit de deuxi\`eme cat\'egorie de Baire. L'op\'erateur $d$ est alors un op\'erateur ferm\'e de domaine de d\'efinition
\[
\{ \varphi\in\A_{K_{\nu_0}}^{2n-p}(\C^n\setminus B)| d\varphi\in \A_c^{2n-p+1}(\C^n\setminus B)\}
\]
entre les espaces de Fr\'echet $\A_{K_{\nu_0}}^{2n-p}(\C^n\setminus B)$ et $\A_c^{2n-p+1}(\C^n\setminus B)\cap d\A_c^{2n-p} (\C^n\setminus B)$ dont l'image est de seconde cat\'egorie de Baire. Le th\'eor\`eme de l'application ouverte implique que cet op\'erateur est surjectif et ouvert (voir par exemple \cite{Samb}, Lemme 3.1). Donc
\[
d\A_{c,K_{\nu_0}}^{2n-p}(\C^n\setminus B)\cap \A_{c,K}^{2n-p+1}(\C^n\setminus B)=\A_{c,K}^{2n-p+1}(\C^n\setminus B)\cap d\A_{c}^{2n-p}(\C^n\setminus B).
\]
Posons $\tilde{K}=K_{\nu_0}$. L'application
\begin{equation*}
\begin{split}
{\rm L}_T^K~: \A_{c,K}^{2n-p+1}(\C^n\setminus B)\cap d\A_{c,\tilde{K}}^{2n-p}(\C^n\setminus B)&\to \C \\
d\varphi & \mapsto \langle{\rm T},\varphi\rangle
\end{split}
\end{equation*}
est bien d\'efinie. En effet, si $d\varphi=d\varphi'$, on a $d(\varphi-\varphi')=0$, $\varphi-\varphi'$ est une $(2n-p)$-forme diff\'erentielle, $d$-ferm\'ee \`a support dans $\tilde{K}$, en particulier dans $\C^n\setminus \bar B$.
\vskip 2mm
Par cons\'equent, il existe $\theta\in\A_{c}^{2n-p-1}(\C^n\setminus B)$  tel que $\varphi-\varphi'=d\theta$. Par densit\'e de $\A_{c}^{2n-p-1}(\C^n\setminus\bar B)$ dans $\A_{c}^{2n-p-1}(\C^n\setminus B)$, il existe une suite $(\theta_j)_{j\in\N}$ d'\'el\'ements de $\A_{c}^{2n-p-1}(\C^n\setminus \bar B)$ qui converge uniforment vers $\theta$ dans $\A_{c}^{2n-p-1}(\C^n\setminus B)$ et par cons\'equent 
\[
\langle T,\varphi\rangle=\langle T,\varphi'\rangle+\langle T,d\theta\rangle=\langle T,\varphi'\rangle
\]
car $T$ \'etant $d$-ferm\'e, 
\[
\langle T,d\theta\rangle=\lim_{j\to+\infty}\langle T,d\theta_j\rangle=0.
\]
Donc
\[
{\rm L}_T^K(d\varphi)={\rm L}_T^K(d\varphi').
\]
\vskip 2mm
L'application ${\rm L}_T^K$ est lin\'eaire et aussi continue comme compos\'ee de deux applications continues (et de la dualit\'e entre $\check{\mathcal{D}}_D^p(\C^n)$ et $\A_c^{2n-p}(\C^n\setminus B)$)~:
\[
T~: \A_{c,\tilde{K}}^{2n-p}(\C^n\setminus B)\to \C
\]
et 
\[
\delta~: \A_{c,K}^{2n-p+1}(\C^n\setminus B)\cap d\big[\A_{c,\tilde{K}}^{2n-p}(\C^n\setminus B)\big]\to \A_{c,\tilde{K}}^{2n-p}(\C^n\setminus B)
\]
qui v\'erifie $d\circ \delta={\rm Id}$ et qui est obtenue par application du th\'eor\`eme de l'application ouverte appliqu\'e \`a 
\begin{equation*}
\begin{split}
d~: \{ \varphi\in \A_{c,\tilde{K}}^{2n-p}(\C^n\setminus B)/\,\ d\varphi\in \A_{c,K}^{2n-p+1}(\C^n\setminus B)\}\subset \A_{c,\tilde{K}}^{2n-p}(\C^n\setminus B)\\
 \qquad \qquad \to \A_{c,K}^{2n-p+1}(\C^n\setminus B)\cap d\big[\A_{c,\tilde{K}}^{2n-p}(\C^n\setminus B)\big].
\end{split}
\end{equation*}
 D'apr\`es le th\'eor\`eme de Hahn-Banach,  on peut {\'e}tendre ${\rm L}_T^K$ en un op{\'e}rateur lin{\'e}aire et continu~:
  \[
  \tilde{\rm L}_T^K~: \A_c^{2n-p+1} (\C^n\setminus B) \to\C
\]
qui est lin\'eaire et continu. Donc $\tilde{{\rm L}}_T^K$ est un courant prolongeable d\'efini dans $\C^n\setminus \bar B$ et $d \tilde{\rm L}_T^K = (- 1)^{2n-p} T$ sur $\stackrel{\circ}{K}$ car ${\rm supp}\varphi\subset K$, 
\[
d\varphi \in\A_{c,K}^{2n-p+1}(\C^n\setminus B)
\]
et 
\[
\langle\tilde{{\rm L}}_T^K,d\varphi\rangle=(-1)^{2n-p}\langle T,\varphi\rangle.
\]
On pose $S^{(K)}=(-1)^{2n-p}\tilde{{\rm L}}_T^K.$
\vskip 1mm
\noindent
D'o{\`u}
  $S^{(K)} = (- 1)^{2n-p} \tilde{{\rm L}}_T^K $ est un courant prolongeable solution de $d u = T$ sur $K$.
\vskip 2mm
\textit{\'Etape 2:}
 Soit maintenant $K_1$, $K_2$ et $K_3$ trois compacts d'int\'erieur non vide de $\C^n\setminus B$ tels que $\stackrel{\circ}{K_1}\subset\subset\stackrel{\circ}{K_2}\subset\subset\stackrel{\circ}{K_3}$ et $\stackrel{\circ}{K_i}\cup \bar B=\{z\in\C^n\,\; \mid z \mid<\eta_i\}$, $i=1,2,3$. Soit $T$ un courant prolongeable sur $\C^n\setminus\bar{B}$ tel qu'il existe $S_2$ et $S_3$ deux $p-1$ courants d\'efinis sur $\stackrel{\circ}{K_2}$ et $\stackrel{\circ}{K_3}$ et prolongeables \`a $\C^n$ tels que, pour tout indice $i=2,3$, $d S_i=T$ sur $\stackrel{\circ}{K_i}$ et soit $\epsilon > 0$, alors il existe un courant prolongeable $\tilde{S_3}$ d\'efini sur $\stackrel{\circ}{K_3}$ tel que~: $d\tilde{S_3}=T$ sur $\stackrel{\circ}{K_3}$ et $\tilde{S_3}_{|\stackrel{\circ}{K_1}}=(S_2)_{|\stackrel{\circ}{K_1}}$ si $2\leq p\leq 2n-1.$
\vskip 1mm
 En effet, comme $dS_2=T$ sur $\stackrel{\circ}{K_2}$ et $dS_3=T$ sur $\stackrel{\circ}{K_3}$, $d(S_2-S_3)=0$ sur $\stackrel{\circ}{K_2}$. Puisque sur $\stackrel{\circ}{K_2}$, on peut r\'esoudre le $d$ pour les formes diff\'erentielles \`a support compact dans $\stackrel{\circ}{K_2}\cup\flat B$ de degr\'e $p$ avec $2\leq 2n-p+1\leq  2n-1$ et $d\big[\A_c^{2n-p-1}(\stackrel{\circ}{K_2}\cup\flat B)\big]$ est ferm\'e dans $\A_c^{2n-p-1}(\stackrel{\circ}{K_2}\cup\flat B)$,  on a d'apr\`es l'\'etape 1 et pour $K$ un compact tel que $\stackrel{\circ}{K_1}\subset\subset K\subset\subset \stackrel{\circ}{K_2}$ un courant $S^{(K)}$ sur $\stackrel{\circ}{K}$ prolongeable \`a $\stackrel{\circ}{K_2}\cup\bar B$ tel que $S_2-S_3=dS^{(K)}$ sur $\stackrel{\circ}{K}$. 
 \vskip 1mm
 Soient $\chi$ une fonction dans $\CC^\infty(\C^n)$ \`a support compact dans $\stackrel{\circ}{K}\cup \bar{B}$ qui vaut $\bf 1$ dans $K_1$ et $\tilde{S}^{(K)}$ une extension de $S^{(K)}$ \`a $\C^n$ 
 \[
 S_3+d(\chi\tilde{S}^{(K)})=S_2-d\big((1-\chi)\tilde{S}^{(K)}\big) \mbox{ sur } \stackrel{\circ}{K_1}.
 \]
 On pose
 \[
 \tilde{S}_3=S_3+d(\chi\tilde{S}^{(K)}).
 \]
\vskip 2mm
 \textit{\'Etape 3:}  
 Consid\'erons une suite exhaustive $(K_j)_{j\in\N}$ de compacts de $\C^n\setminus B$. Supposons que $\stackrel{\circ}{K_j}\cup\bar B=\{z\in\C^n\,\: \mid z \mid<\eta_j\}$ o\`u $(\eta_j)_{j\in\N}$ sont des r\'eels tels que $\eta_j <\eta_{j+1}$. Pour $2\leq p\leq 2n-1$, on associe \`a $(K_j)_{j\in \N}$ gr\^ace aux \'etapes 1 et 2 une suite de courants $(S_j)_{j\in\N}$ d\'efinis dans $K_j$ et prolongeables \`a $\C^n$ telle que $d S_j=T$ sur $\stackrel{\circ}{K_j}$ et si $j$, $j+1$, $j+2$ sont trois indices cons\'ecutifs, $S_{j+2}=S_{j+1}$ sur $\stackrel{\circ}{K_j}$.
 \vskip 1mm
 \noindent
 La suite $(S_j)_{j\in N}$ converge vers un courant $S$ d\'efini sur $\C^n\setminus \bar{B}$ et prolongeable. De plus, $S$ est solution de l'\'equation $du=T$ dans $\C^n\setminus\bar{B}$.
 \end{proof}
\section{R\'esolution du $\partial\bar\partial$ pour les courants prolongeables}
Tenant compte du th{\'e}or{\`e}me \ref{thm1} et des r{\'e}sultats de r{\'e}solution du
$\bar{\partial}$ pour les courants prolongeables obtenus par S. Sambou dans \cite{Samb}, on a le th{\'e}or{\`e}me suivant~:
\begin{theorem}\label{thm3}$~~ $\\
Soit $T$ un $(p,q)$-courant prolongeable d{\'e}fini sur $\C^n\setminus \bar{B}$. Supposons que $d T = 0$; $1\leqslant p \leqslant n$ et $1 \leqslant q \leqslant n$, alors il existe un
$(p-1,q-1)$-courant $S$ d{\'e}fini sur $\C^n\setminus \bar{B}$, prolongeable tel que
$\partial \bar{\partial} S = T$, pour $2 \leqslant p + q \leqslant 2n-1$.
\end{theorem}
\begin{proof}
  Soit $T$ un $(p, q)$-courant, $1 \leqslant p \leqslant n$ et $1 \leqslant
  q \leqslant n$, $d$-ferm{\'e} d{\'e}fini sur $\C^n\setminus \bar{B}$ et prolongeable
  avec $2 \leqslant p + q \leqslant 2n-1$.
  \vskip 1mm
  \noindent
 Puisque le th{\'e}or{\`e}me \ref{thm1} nous assure que $\check{{\rm H}}^{p + q} (\C^n\setminus \bar{B})=0$, il existe un courant prolongeable $\mu$ d{\'e}fini sur $\C^n\setminus \bar{B}$ tel que
  $d \mu = T$. $\mu$ est un $(p + q - 1)$-courant, il se d{\'e}compose en un $(p - 1,
  q)$-courant $\mu_1$ et en un $(p, q - 1)$-courant $\mu_2$. On a
  \[
  d \mu = d (\mu_1 +
  \mu_2) = d \mu_1 + d \mu_2 = T.
  \]
Comme $d=\partial+\bar\partial$, on a, pour des raisons de bidegr\'e, $\partial\mu_2=0$ et $\bar\partial\mu_1=0$. On obtient $\mu_1=\bar\partial u_1$ et $\mu_2=\partial u_2$ avec $u_1$ et $u_2$ des courants prolongeables d\'efinis sur $\C^n\setminus \bar{B}$, (voir \cite{Samb}, section 3).
\vskip 1mm
  On a :
\begin{eqnarray*}
T &= &\partial \mu_1  + \bar{\partial} \mu_2\\
&=&\partial \bar{\partial} u_1 + \bar{\partial}
  \partial u_2\\
  &=&\partial \bar{\partial} (u_1 - u_2)
\end{eqnarray*}
   Posons $S = u_1 - u_2$ , $S$ est un $(p - 1, q - 1)$-courant prolongeable
  d{\'e}fini sur $\C^n\setminus \bar{B}$ tel que $\partial \bar{\partial} S = T$.
\end{proof}

\section{R\'esolution du $\partial\bar\partial$ sur une vari\'et\'e analytique complexe}

On va maintenant consid\'erer $X$ comme une vari\'et\'e diff\'erentiable de dimension $n$.
\begin{definition}
\end{definition}
Soit $X$ une vari\'et\'e diff\'erentiable de dimension $n$ et $\omega\subset X$ un domaine contractile. On dit que $X$ est une extension contractile de $\Omega$, s'il existe une suite $\left(\Omega_{n} \right)_{n} $ exhaustive de domaines contractiles telle que $$ \forall n, \Omega\subset\subset \Omega_{n}\subset\subset X. $$
\begin{example}
\end{example}
Quand $X=\mathbb{C}^{n}$, alors $\mathbb{C}^{n}$ est une extension contractile de la boule unit\'e $B$.\\
\\
On a pour les extensions contractiles, le th\'eor\`eme suivant:
\begin{theorem}\label{thm4}$~~ $\\
Soit $X$ une vari\'et\'e analytique complexe de dimension $n$ et soit $D\subset\subset X$ un domaine contractile fortement pseudoconvexe. Supposons que $X$ est une extension $(n-1)$-convexe avec $ H^{j}(\flat D)=0 \quad 2\leqslant j\leqslant 2n-2$ de $D$ et une extension contractile de $D$. Posons $\Omega=X\setminus\bar{D}$. Si $\stackrel{\circ}{\bar{\Omega}}=\Omega$, alors pour tout $(p,q)$ courant $T$ d\'efini sur $\Omega$, $d$-ferm\'e et prolongeable, il existe un $(p-1,q-1)$ courant $S$ d\'efini sur $\Omega$ et prolongeable tel que $\partial\bar{\partial}S=T$ pour $1\leqslant p\leqslant n-1$ et $1\leqslant q\leqslant n-1$.
\end{theorem}

Pour d\'emontrer le th\'eor\`eme \ref{thm4}, nous avons besoin du lemme suivant~:
\begin{lemma}
\[
\A_c^r(\bar\Omega)\cap\ker d=d\big(\A_c^{r-1}(\bar\Omega)\big) \mbox{ pour } 1\leq r \leq 2n-1.
\]
\end{lemma}

\begin{proof}$~~ $\\
On a $X=\cup D_{\nu} \quad , D\subset\subset D_{\nu}\subset\subset X$ et $D_{\nu}$ est contractile.\\
Si $f\in\A_c^r(\bar{\Omega)}\cap\ker d \quad \exists \nu_{0}\in \mathbb{N}$ tel que $f\in \A_c^r(D_{\nu_{0}})\cap \ker d$. Or $ H^{j}(D_{\nu_{0}})=0 ,\quad \textrm{pour} j\geq 1 $. Par dualit\'e de Poincar\'e $H^{j}_{c}(D_{\nu_{0}})=0 \quad \textrm{pour} j<2n$. Il existe $g\in \A_c^{r-1}(D_{\nu_{0}})$, donc $g\in \A_c^{r-1}(X)$ telle que $dg=f$.\\
$dg_{\mid D}=0$, si $r=1$ alors $g$ est une constante sur $D$. Si $r>1$, il existe $u$ une $(r-2)$ forme diff\'erentielle sur $\bar{D}$ telle que $ g_{\mid D}=du $. Soit $\tilde{u}$ une extension de $u$ \`a support compact dans $X$, $h=g-d\tilde{u}$ convient.
\end{proof}

Nous pouvons \'etablir donc la preuve du th\'eor\`eme \ref{thm4}
\begin{proof}[D\'emonstration du th\'eor\`eme \ref{thm4}] $~~ $\\
%$X$ est contracile, donc 
%\[
%{\rm H}^{p,q}(X)=0, 1\leq p+q.
%\]
% Si $\stackrel{\circ}{\bar\Omega} = \Omega$, alors 
%\[
%\check{\mathcal{D}}'^r(\Omega)=\big[\mathcal{D}^{2n-r}(\bar\Omega)\big]'.
%\]
\vskip 1mm
Soit une suite exhaustive de compacts $K_{\nu}$ de $\Omega$. $$ K_{\nu}=\bar{D_{\nu}}\setminus \bar{D} $$
et quelque soit $\nu$, $K_\nu$ est un compact d'int\'erieur non vide. 
\vskip 1mm
\noindent
Pour $K$ un compact de $\Omega$, l'op\'erateur ${\rm L}_T^K$ est bien d\'efini, lin\'eaire et continu, cf. \'etape 1 de la d\'emonstrations du th\'eor\`eme \ref{thm1}.
\vskip 2mm
\textit{\'Etape 1:}\\
 D'apr\`es le th\'eor\`eme de Hahn-Banach,  on peut {\'e}tendre ${\rm L}_T^K$ en un op{\'e}rateur lin{\'e}aire et continu~:
  \[
  \tilde{\rm L}_T^K~: \A_c^{2n-p+1} (\Omega) \to\C
\]
qui est lin\'eaire et continu. Donc $\tilde{{\rm L}}_T^K$ est un courant prolongeable d\'efini dans $ \bar\Omega$ et $ d\tilde{\rm L}_T^K = (- 1)^{2n-p} T$ sur $\stackrel{\circ}{K}$ car ${\rm supp}\varphi\subset K$, 
\[
d\varphi \in\A_{c,K}^{2n-p+1}(\Omega)
\]
et 
\[
\langle\tilde{{\rm L}}_T^K,d\varphi\rangle=(-1)^{2n-p}\langle T,\varphi\rangle.
\]
On pose $S^{(K)}=(-1)^{2n-p}\tilde{{\rm L}}_T^K.$
\vskip 1mm
\noindent
D'o{\`u}
  $S^{(K)} = (- 1)^{2n-p} \tilde{{\rm L}}_T^K $ est un courant prolongeable solution de $d u = T$ sur $K$.
\vskip 2mm
\textit{\'Etape 2:}\\
 Soit maintenant $K_1$, $K_2$ et $K_3$ trois compacts d'int\'erieur non vide de $\Omega$ tels que $\stackrel{\circ}{K_1}\subset\subset\stackrel{\circ}{K_2}\subset\subset\stackrel{\circ}{K_3}$ et $\stackrel{\circ}{K_i}\cup \bar D=\{z\in X\,\ ; |z|<\eta_i\}$, $i=1,2,3$. Soit $T$ un courant prolongeable sur $\bar{\Omega}$ tel qu'il existe $S_2$ et $S_3$ deux $(p-1)$ courants d\'efinis sur $\stackrel{\circ}{K_2}$ et $\stackrel{\circ}{K_3}$ et prolongeables \`a $X$ tels que, pour tout indice $i=2,3$, $d S_i=T$ sur $\stackrel{\circ}{K_i}$, alors il existe un courant prolongeable $\tilde{S_3}$ d\'efini sur $\stackrel{\circ}{K_3}$ tel que~: $d\tilde{S_3}=T$ sur $\stackrel{\circ}{K_3}$ et  $\tilde{S_3}_{|\stackrel{\circ}{K_1}}=(S_2)_{|\stackrel{\circ}{K_1}}$ si $2\leq p.$
 \vskip 1mm
 En effet, comme $d S_2=T$ sur $\stackrel{\circ}{K_2}$ et $d S_3=T$ sur $\stackrel{\circ}{K_3}$, $d(S_2-S_3)=0$ sur $\stackrel{\circ}{K_2}$. Puisque sur $\stackrel{\circ}{K_2}$, on peut r\'esoudre le $d$ pour les formes diff\'erentielles \`a support compact dans $\stackrel{\circ}{K_2}\cup\,\ \flat D$ de degr\'e $p$ avec $2\leq 2n-p+1\leq 2n-1$ et $d\big[\A_c^{2n-p-1}(\stackrel{\circ}{K_2}\cup\,\ \flat D)\big]$ est ferm\'e dans $\A_c^{2n-p-1}(\stackrel{\circ}{K_2}\cup\,\ \flat D)$, on a d'apr\`es l'\'etape 1 et pour $K$ un compact tel que $\stackrel{\circ}{K_1}\subset\subset K\subset\subset \stackrel{\circ}{K_2}$ un courant $S^{(K)}$ sur $\stackrel{\circ}{K}$ prolongeable \`a $\stackrel{\circ}{K_2}\cup\bar D$ tel que $S_2-S_3=d S^{(K)}$ sur $\stackrel{\circ}{K}$. 
 \vskip 1mm
 Soient $\chi$ une fonction de classe $\CC^\infty$ \`a support compact dans $\stackrel{\circ}{K}\cup \bar{D}$ qui vaut $1$ dans $K_1$ et $\tilde{S}^{(K)}$ une extension de $S^{(K)}$ \`a $X$ 
 \[
 S_3+d(\chi\tilde{S}^{(K)}=S_2-d\big((1-\chi)\tilde{S}^{(K)}\big) \mbox{ sur } \stackrel{\circ}{K_1}.
 \]
 On pose
 \[
 \tilde{S}_3=S_3+d(\chi\tilde{S}^{(K)}).
 \]
\vskip 2mm
 \textit{\'Etape 3:} \\ 
 Consid\'erons une suite exhaustive $(K_j)_{j\in\N}$ de compacts de $\Omega$. Supposons que $\stackrel{\circ}{K_j}\cup\bar D=\{z\in X\,\: \mid z \mid<\eta_j\}$ o\`u $(\eta_j)_{j\in\N}$ sont des r\'eels tels que $\eta_j <\eta_{j+1}$. Pour $2\leq  p$, on associe \`a $(K_j)_{j\in \N}$ gr\^ace aux \'etapes 1 et 2 une suite de courants $(S_j)_{j\in\N}$ d\'efinis dans $K_j$ et prolongeables \`a $X$ telle que $d S_j=T$ sur $\stackrel{\circ}{K_j}$ et si $j$, $j+1$, $j+2$ sont trois indices cons\'ecutifs, $S_{j+2}=S_{j+1}$ sur $\stackrel{\circ}{K_j}$.
 \vskip 1mm
 \noindent
 La suite $(S_j)_{j\in N}$ converge vers un courant $S$ d\'efini sur $\Omega $ et prolongeable \`a $X$. De plus, $S$ est solution de l'\'equation $d u=T$ dans $\Omega$.
 \vskip 2mm
 \textit{\'Etape 4 :} \\
 Soit $T$ un $(p, q)$-courant, $1 \leqslant p \leqslant n$ et $1 \leqslant
  q \leqslant n$, $d$-ferm{\'e} d{\'e}fini sur $\Omega$ et prolongeable
  avec $2 \leqslant p + q \leqslant 2n-2$.
  \vskip 1mm
  \noindent
 Puisque le th{\'e}or{\`e}me \ref{thm1} nous assure que $\check{{\rm H}}^{p + q} (\Omega)=0$, il existe un courant prolongeable $\mu$ d{\'e}fini sur $\Omega$ tel que
  $d \mu = T$. $\mu$ est un $(p + q - 1)$-courant, il se d{\'e}compose en un $(p - 1,
  q)$-courant $\mu_1$ et en un $(p, q - 1)$-courant $\mu_2$. On a
  \[
  d \mu = d (\mu_1 +
  \mu_2) = d \mu_1 + d \mu_2 = T.
  \]
Comme $d=\partial+\bar\partial$, on a, pour des raisons de bidegr\'e, $\partial\mu_2=0$ et $\bar\partial\mu_1=0$. On obtient $\mu_1=\bar\partial u_1$ et $\mu_2=\partial u_2$ avec $u_1$ et $u_2$ des courants prolongeables d\'efinis sur $\Omega$, (voir \cite{Samb}, section 3).
\vskip 1mm
  On a~:
\begin{eqnarray*}
T &= &\partial \mu_1  + \bar{\partial} \mu_2\\
&=&\partial \bar{\partial} u_1 + \bar{\partial}
  \partial u_2\\
  &=&\partial \bar{\partial} (u_1 - u_2)
\end{eqnarray*}
   Posons $S = u_1 - u_2$ , $S$ est un $(p - 1, q - 1)$-courant prolongeable
  d{\'e}fini sur $\Omega$ tel que $\partial \bar{\partial} S = T$.
 
\end{proof}

\end{document}